\documentclass{amsart}

\usepackage{amsfonts}

\usepackage{amsthm}
\usepackage{natbib}

\usepackage{setspace}
\usepackage{color}
\usepackage{breakcites}
\usepackage[utf8]{inputenc}

\usepackage{xcolor}
\usepackage{amsmath}
\usepackage{amssymb}
\usepackage{comment}

\usepackage{graphicx}
\usepackage{subfigure}
\usepackage{adjustbox}
\usepackage{bm}
\usepackage{geometry}
\usepackage{tikz}
\usetikzlibrary{decorations.pathreplacing}
\usetikzlibrary{fadings}

\allowdisplaybreaks

\author{Fred Espen Benth, Nils Detering, Luca Galimberti}
\address{Fred Espen Benth \\
University of Oslo\\
Department of Mathematics \\
P.O. Box 1053, Blindern\\
N--0316 Oslo, Norway}
\email[]{fredb\@@math.uio.no}
\address{Nils Detering \\ 
Heinrich Heine University D\"usseldorf\\
Department of Mathematics\\
Universit\"atsstrasse 1, 40225 D\"usseldorf, Germany}
\email[]{nils.detering\@@hhu.de}
\address{Luca Galimberti \\ 
King's College London\\
Department of Mathematics\\
Strand Building, WC2R 2LS, London, UK}
\email[]{luca.galimberti\@@kcl.ac.uk}

\newtheorem{theorem}{Theorem}[section]
\newtheorem{definition}[theorem]{Definition}

\newtheorem{proposition}[theorem]{Proposition}
\newtheorem{assumption}[theorem]{Assumption}
\newtheorem{corollary}[theorem]{Corollary}

\newtheorem{remark}[theorem]{Remark}

\newcommand{\R}{\mathbb R}

\newcommand{\N}{\mathbb N}

\newcommand{\EE}{\mathbb E} 

\newcommand{\norm}[1]{\left\lVert#1\right\rVert}
\newcommand{\abs}[1]{\left |#1\right|}
\newcommand{\X}{\mathfrak X}

\DeclareMathOperator{\Span}{span}

\newcommand{\ft}{F^t}
\newcommand{\fth}{F^{t,K,\phi}}

\date{\today}

\title[Structure-informed operator learning for parabolic PDEs]{Structure-informed operator learning for parabolic Partial Differential Equations}

\begin{document}
\begin{abstract}
In this paper, we present a framework for learning the solution map of a backward parabolic Cauchy problem. The solution depends continuously but nonlinearly on the final data, source, and force terms, all residing in Banach spaces of functions. We utilize Fr\'echet space neural networks (\cite{FrechetNN}) to address this operator learning problem. Our approach provides an alternative to Deep Operator Networks (DeepONets), using basis functions to span the relevant function spaces rather than relying on finite-dimensional approximations through censoring. With this method, structural information encoded in the basis coefficients is leveraged in the learning process. This results in a neural network designed to learn the mapping between infinite-dimensional function spaces. Our numerical proof-of-concept demonstrates the effectiveness of our method, highlighting some advantages over DeepONets.
\end{abstract}
\maketitle 

\section{Introduction}
In this paper we provide a 
framework for learning the solution map of a backward parabolic Cauchy problem incorporating structural information about the final datum, source, and the force term. The solution of the Cauchy problems depends in a nonlinear, but continuous way on the final datum, the source and force terms, which are all functions 
living in appropriate Banach spaces. We propose to use the recently introduced neural networks in Fr\'echet spaces (see \cite{FrechetNN}), an infinite dimensional neural network structure,  to solve 
this operator learning problem.

Our approach provides an alternative to the operator learning method Deep Operator Networks (DeepONets), studied by \cite{LuLu-arxiv} and \cite{deepOnet}. DeepONet extends the shallow network for operator learning proposed and analysed in \cite{ChenChen}. The validity of these operator learning approaches is resting on the universal approximation theorem (see \cite{ChenChen}, \cite{deepOnet} and more recently a generalisation  by \cite{Lanthaler} to measurable operators).  Instead of using finite-dimensional neural networks approximating sampled (also called censored) expressions of the input functions in the operator in question, as done in DeepONets, we 
make use of the information contained in the basis functions spanning the relevant function spaces. We build a neural network which is learning the map between these function spaces expressed  by their basis functions. An infinite dimensional activation function allows us to set up a deep neural network that preserves the structural information encoded in the basis coefficients when processing through the layers.

Our approach rests on a truly infinite dimensional neural network, reflecting that we are approximating continuous nonlinear operators between infinite dimensional spaces. Implemented on a computer, we sample a finite set of basis functions in the training rather than censoring the input functions to have finite dimensional approximations. We refer also to \cite[Sect. 2, p. 9]{anandkumar-fourier2} where a similar idea was mentioned but not further explored.

Often linear operators between function spaces can be expressed as integral operators, for example as convolutions. An approximation of such integral operators using graph kernels is proposed in 
\cite{anandkumar-kernel} for operator learning of partial differential equations. The authors take an infinite dimensional perspective in learning operator maps from various parameters into the solution, viewed as continuous mappings between function spaces of Sobolev type. The affine transform-part of the neural network is viewed as an integral kernel operator, and a discrete version of this is mapped to the next layer by a finite dimensional activation function. These ideas are further expanded into Fourier neural operators (see \cite{anandkumar-fourier1, anandkumar-fourier2}, where the kernel is represented by the Fourier transform and its inverse to obtain computationally attractive representations for estimating the kernel operator. \cite{LaplaceNO-karniadakis} propose to use the Laplace transform instead of the Fourier transform, taking advantage of the pole-residue relationship between the input and output space of the operator to be approximated.
In \cite{anandkumar-fourier2} a universal approximation theorem is shown for networks learning nonlinear operators between certain Banach spaces of functions with infinite dimensional (integral operator) affine transforms, but finite dimensional activation functions. In a recent paper \cite{anandkumar-PINN-operator} the Fourier neural operator methodology is applied in conjunction with physics informed learning. We refer the reader to the extensive literature review on operator and physics informed learning learning in \cite{anandkumar-PINN-operator} (see also more references in \cite{LaplaceNO-karniadakis}).

In \cite{FrechetNN} a universal approximation theorem was shown which ensures that continuous operators from a Fr\'echet space into a Banach space can be approximated on compacts with Fr\'echet neural networks. In our analysis of the parabolic Cauchy problem, using the classical theory of e.g. \cite{friedman1975}, the solution map can be shown to be Lipschitz continuous from a product of Sobolev spaces into the continuous functions on compacts. Thus, the universal approximation theorem ensures that we can approximate the operator arbitrary well by the Fr\'echet neural network. 
In \cite{Lanthaler} the operators they are interested in training using DeepONet are shown to be Lipschitz continuous.
For the Fourier neural operators studied in \cite{anandkumar-fourier2}, proofs of the integrability properties of the operators seem to be missing for the operators they aim to train that enables them to use of their universal approximation theorem. 
\cite{LaplaceNO-karniadakis} do not address the regularity properties for the operators they study numerically.

In a numerical proof-of-concept, we demonstrate that our proposed approach indeed works out. We compare with the DeepONets, and point out some advantages with our approach. 
We once again would like to emphasise also that we make use of infinite dimensional activation functions, and not finite dimensional ones as in DeepONet and the Fourier neural operator approaches, because for deep learning this allows us to progress structural information from the basis functions throughout the layers.     

Our approach and analysis extends the neural network approach to approximate numerically high dimensional partial differential equations by training using synthetic data generated by stochastic differential equations along with the Feynman-Kac formula, as advocated by \cite{E-Jentzen} and \cite{Jentzen} (see also \cite{Jentzen-survey} for an overview and further references). We can naturally make use of the Feynman-Kac formula also for operator learning problems related to Cauchy problems, as we show in this paper. In passing we remark that
\cite{FS-paper} have made use of similar ideas to price options in energy markets which require an infinite dimensional framework.

Learning the operator mapping certain parameter functions into the solution of partial differential equations is a forward problem, often referred to as ``many query''. A neural operator method allows for fast and efficient computations of the solution for various specifications of the input parameter functions (see \cite{Lanthaler} for a discussion and analysis of the curse of dimensionality in such problems). 
The inverse problem, where observations of the dynamical system in question is available and one wants to back out parameters, is also of interest, and has been empirically investigated in a Bayesian framework in \cite{anandkumar-fourier1}. Rather than learning the operator map for a specific dynamical system, \cite{osher} propose a framework to learn operators connected to a family of differential equations. Empirical evidence demonstrates that commonalities across differential equations reduce the training burden in such an approach.

Our results are presented as follows. In the next section we provide a review of infinite dimensional neural network introduced in \cite{FrechetNN}, collecting some useful material. Section \ref{sect:cont} defines the Cauchy problem relying on the classical analysis of \cite{friedman1975}, and identifies the operator maps that will be the core object of analysis in this paper. To use infinite dimensional neural networks to learn the operator maps, we need continuity properties to hold according to 
the universal approximation theorem. Continuity of the nonlinear operator maps are analysed and shown for Sobolev spaces in this Section. A numerical example is given in Section \ref{sect:numerics}, providing a proof-of-concept. Here we are
benchmarking our proposed method with the DeepONets-approach, and provide further extensions and perspectives of our proposed structure-informed operator learning approach.

\section{A brief introduction to neural networks in infinite dimensions}
\label{NN:Intro}

In this Section, we give a brief review of neural networks defined on infinite dimensional spaces, 
following the approach in \cite{FrechetNN}. Although \cite{FrechetNN} consider networks on Fr\'echet spaces, we focus on the case of real Banach spaces in the account given here. 

Let $\mathfrak{X}$ and $\mathfrak{Y}$ be two real Banach spaces with norms denoted by $\Vert\cdot\Vert_{\mathfrak{X}}$ and $\Vert\cdot\Vert_{\mathfrak{Y}}$, resp. We are interested in learning continuous nonlinear operators $F:\mathfrak{X}\rightarrow\mathfrak{Y}$ by neural networks. We denote $C(\mathfrak{X};\mathfrak{Y})$ the space of such continuous operators, equipped with the topology of uniform convergence on compacts. If $\mathfrak{Y}=\mathbb R$, we use the simpler notation $C(\mathfrak{X})$ to denote $C(\mathfrak{X};\mathbb R)$.  

Let us start with defining a one layer real-valued neural network on
$\mathfrak{X}$. Let $A\in\mathcal L(\mathfrak{X})$, i.e. a linear and continuous operator $A:\mathfrak{X}\rightarrow\mathfrak{X}$, and $\beta\in\mathfrak{X}$. Thus, $\mathfrak{X}\ni \xi\rightarrow A\xi+\beta\in\mathfrak{X}$ is an {\it affine} transform on $\mathfrak{X}$. Introduce an {\it activation function}  $\sigma:\mathfrak{X}\rightarrow\mathfrak{X}$ being continuous, and define a {\it neuron} $\mathcal N_{\ell,A,\beta}:\mathfrak{X}\rightarrow\mathbb R$ to be the map 
\begin{equation}
    \mathcal N_{\ell,A,\beta}(\xi)=\ell(\sigma(A\xi+\beta))
\end{equation}
for $\ell\in\mathfrak{X}'$. Here, $\mathfrak{X}'$ is the (topological) dual of $\mathfrak{X}$, i.e., the space of continuous linear functionals $\ell:\mathfrak{X}\rightarrow\mathbb R$. A one-layer neural network of width $M\in \mathbb N$ is given as a linear superposition of $M$ such neurons,
\begin{equation}
\label{nn-def-real}
\mathcal N(\xi):=\sum_{j=1}^M\mathcal N_{\ell_j,A_j,\beta_j}(\xi).
\end{equation}

As activation functions $\sigma$, we restrict our attention to the "sigmoidal" class. As we operate with infinite dimensional activation functions, the sigmoidal property requires some care. Here is \cite[Def. 2.6]{FrechetNN}, where a separation property is introduced for continuous $\sigma:\mathfrak{X}\rightarrow\mathfrak{X}$:
\begin{definition}[Separation property]
There exist $\psi\in\mathfrak{X}'\backslash\{0\}$ and $\nu_+, \nu_{-}, \nu_0\in\mathfrak{X}$ 
such that either
$\nu_{+}\notin\text{span}\{\nu_0,\nu_{-}\}$
or 
$\nu_{-}\notin\text{span}\{\nu_0,\nu_{+}\}$
and such that
\begin{align*}
&\lim_{\lambda\rightarrow\infty}\sigma(\lambda \xi)=\nu_+, \text{ if } \xi\in\Psi_+ \\
&\lim_{\lambda\rightarrow\infty}\sigma(\lambda \xi)=\nu_{-}, \text{ if } \xi\in\Psi_{-} \\ 
&\lim_{\lambda\rightarrow\infty}\sigma(\lambda \xi)=\nu_0, \text{ if } \xi\in\Psi_0,
    \end{align*}
where $\Psi_+:=\{\xi\in\mathfrak{X}\,\vert\,\psi(\xi)>0\}$, $\Psi_{-}:=\{\xi\in\mathfrak{X}\,\vert\,\psi(\xi)<0\}$ and $\Psi_0:=\ker(\psi)$.
\end{definition}
We also need the activation functions to be bounded, in the sense that its image $\sigma(\mathfrak{X})$ is a bounded subset of $\mathfrak{X}$. Indeed, one sees that with continuous and bounded activation functions, the one-layer neural networks that we have introduced are also continuous and bounded. 

From \cite[Example 2.13]{FrechetNN} we find a specification of a class of bounded and continuous activation functions which enjoy the separation property. Let us recall this example: fix $\psi\in\mathfrak{X}'\backslash\{0\}$ and suppose we have a sequence $(\widetilde{\sigma}_j)_{j\in\mathbb N}$ of continuous functions
$\widetilde{\sigma}_j:\mathbb R\rightarrow\mathbb R$ with the properties
$\widetilde{\sigma}_j(0)=0$,
$$
\lim_{x\rightarrow\infty}\widetilde{\sigma}_j(x)=1,\qquad \lim_{x\rightarrow-\infty}\widetilde{\sigma}_j(x)=0,
$$
for all $j\in\mathbb N$, where additionally we suppose $\sup_{j\in\mathbb N}\Vert\widetilde{\sigma}_j\Vert_{\infty}<\infty$. Given a summable sequence
$(\zeta_j)_{j\in\mathbb N}\subset\mathfrak{X}$ for which $0\neq\zeta:=\sum_{j=1}^{\infty}\zeta_j$, define the function $\sigma:\mathfrak{X}\rightarrow\mathfrak{X}$ by
\begin{equation}\label{eq: summable sequence}
\sigma(\xi)=\sum_{j=1}^{\infty}\widetilde{\sigma}_j(\psi(\xi))\zeta_j
\end{equation}
One can show (see \cite[Example 2.13]{FrechetNN}) that $\sigma$ is continuous, bounded and separating, and thus suitable as an activation function. Notice that $\sigma$ is a sum of $\widetilde{\sigma}_j(\psi(\xi))\zeta_j$, where $\widetilde{\sigma}_j$ are classical sigmoidal activation functions. We first let a global linear functional act on $\xi$ before it goes as input into the activation function
$\widetilde{\sigma}_j$, having a real value. We use this real value to scale the element $\zeta_j$ in the Banach space to build up an activation function which truly maps $\mathfrak{X}$ into itself. In a simple setting, one can choose only one such element in the series (or a finite number of such), defining the activation function 
$\sigma(\xi):=\widetilde{\sigma}(\psi(\xi))\zeta$. We mention in passing that there exist other examples of infinite dimensional
activation functions
(see \cite{FrechetNN}). 

To obtain neural networks with values in the Banach space $\mathfrak{Y}$, we simply aggregate the real-valued neural networks we
have defined in \eqref{nn-def-real}, 
scaled by independent unit vectors in $\mathfrak{Y}$. To be more precise, given $d$ independent unit vectors $\mu_1,\ldots,\mu_d\in\mathfrak{Y}$ and neural networks $\mathcal N^{(1)},\ldots,\mathcal N^{(d)}$, define the $\mathfrak{Y}$-valued neural network $\mathcal N_d:\mathfrak{X}\rightarrow\mathfrak{Y}$ as
\begin{equation}
    \label{nn-def-inf}
    \mathcal N_d(\xi)=\sum_{i=1}^d\mathcal N^{(i)}(\xi)\mu_i. 
\end{equation}
For such neural networks we have the following {\it Universal Approximation Theorem} (see \cite[Thm. 3.2]{FrechetNN}):
\begin{theorem}
    \label{thm:UAT}
    Assume $\sigma:\mathfrak{X}\rightarrow\mathfrak{X}$ is a continuous, bounded and separating activation function and suppose that $F\in C(\mathfrak{X};\mathfrak{Y})$. Then, for given compact set $\mathcal K\subset\mathfrak{X}$ and $\epsilon>0$ there exist $d\in\mathbb N$, $d$ independent unit vectors $\mu_1,\ldots,\mu_d\in\mathfrak{Y}$ and $d$ real-valued neural networks $\mathcal N^{(1)},\ldots,\mathcal N^{(d)}$ such that
    $$
    \sup_{\xi\in\mathcal K}\Vert F(\xi)-\mathcal N_d(\xi)\Vert_{\mathfrak{Y}}<\epsilon,
    $$
    where $\mathcal N_d$ is defined in \eqref{nn-def-inf}.
\end{theorem}
This universal approximation theorem is the key to the applicability of infinite dimensional neural networks in operator-learning tasks. Remark that if we let $\mathfrak{Y}=\mathbb R$, then we can re-state the universal approximation theorem as follows: for given $\epsilon>0$ and compact subset $\mathcal K\subset\mathfrak{X}$, there exists $M\in\mathbb N$ such that
$$
\sup_{\xi\in\mathcal K}\vert F(\xi)-\mathcal N(\xi)\vert<\epsilon,
$$
where $\mathcal N$ is defined in \eqref{nn-def-real}. By saying that there exists an $M$, we are in reality also saying that there exist elements $\ell_j\in\mathfrak{X}'$, $A_j\in\mathcal L(\mathfrak{X})$ and $\beta_j\in\mathfrak{X}$ 
for $j=1,\ldots,M$. Obviously, all these elements, including $M$, depend on $\epsilon$ and $\mathcal K$.  

To implement the neural network $\mathcal N_d$ on a computer we need a finite dimensional version thereof. In dealing with partial differential equations and operator learning tasks, Sobolev spaces appear naturally. These Sobolev spaces are in most cases separable Banach spaces and carry a Schauder basis, which one can exploit to create finite dimensional networks from 
our infinite dimensional $\mathcal N_d$. Indeed, as basis functions contain structural information about the functions that go into the operator, one can use this information in the training. We provide some more details on this idea which is key to our method. 

Suppose $\mathfrak{X}$ is a {\it separable} Banach space, with Schauder basis denoted by $(e_k)_{k\in\mathbb N}$. Thus, for any $\xi\in\mathfrak X$ we have unique coefficients $a_k\in\R,\,k=1,2,\dots$ such that  
$\xi=\sum_{k=1}^{\infty}a_k e_k$. Without loss of generality, we can assume that $\Vert e_k\Vert_\X=1$ for all $k$. We define the canonical linear continuous projectors
\[
p_k:\X\to \R,\quad \xi\mapsto a_k,\quad k\in \N,
\]
 and the projection operators 
\begin{equation}
    \Pi_N:\mathfrak X\rightarrow\text{span}\{e_1,\ldots,e_N\},\qquad \xi\mapsto \sum_{k=1}^Np_k(\xi)e_k,
\end{equation}
for $N\in\mathbb N$. The projection operators are also linear, bounded, and $\Pi_N$ converges uniformly on compacts $\mathcal K\subset\mathfrak{X}$ when $N\rightarrow\infty$ (see \cite[Thm. 9.6, p. 115]{Schaefer}). 

Given a network $\mathcal N$ as in \eqref{nn-def-real} and $N\in\mathbb N$, define a finite dimensional network $\mathcal N_{N}$ as 
\begin{equation}
    \mathcal N_{N}(\xi)=\sum_{j=1}^M(\ell_j\circ\Pi_N)(\sigma(\Pi_NA_j\Pi_N\xi+\Pi_N\beta_j)).
\end{equation}
Let us briefly explain the reason why this architecture can be seen as a finite dimensional object: first of all, for any $N\in\N$, $\Pi_N\xi\in\X$ can be identified with its truncated expansion $(p_1(\xi),\dots, p_N(\xi))\in\R^N$. Secondly, the restrictions of the operators $\Pi_N\circ A_j$ and $\ell_j$ to $\Span\{e_1,\dots,e_N\}$ are finite dimensional, because the action of $\ell_j$ will be prescribed by the scalars $\ell_j(e_1),\dots,\ell_j(e_N)$, and the action of $\Pi_N\circ A_J\circ\Pi_N$ will be specified by the matrix $\left\{ p_m(A_je_k)\right\}_{m,k=1}^N$. The sum above thus resembles a classical neural network. However, instead of the typical one dimensional activation function, the function $\Pi_N \circ \sigma$ restricted to $\Span\{e_1,\dots ,e_N\}$ is multidimensional. The terms appearing in the sum can now easily be programmed in a computer: we refer to Section \ref{sect:numerics} for further details.

Assuming additionally that the activation function $\sigma$ is Lipschitz (which holds if $\widetilde{\sigma}_j$ is Lipschitz for each $j\in\mathbb N$ and the Lipschitz constants satisfy a uniform bound $\sup_{j\in\mathbb N} Lip(\widetilde{\sigma}_j)<\infty$: see \eqref{eq: summable sequence}), then the Universal Approximation Theorem is valid for $\mathcal N_{N}$ (see \cite[Prop. 4.1]{FrechetNN}).

The networks $\mathcal N$ and $\mathcal N_{N}$ above can also be extended to deep neural networks. An infinite dimensional network with $n\in\mathbb N$ layers can be constructed from neurons of the form
\begin{equation}
    \mathcal N_{\ell,\mathbf{\mathcal A}}(\xi):=\ell(\sigma\circ \mathcal A_1\circ\cdots\circ\sigma\circ \mathcal A_n)(\xi)
\end{equation}
with $\mathbf{\mathcal A}=(\mathcal A_1,\ldots,\mathcal A_n)$ a vector of affine transformation on $\mathfrak{X}$ given as $\mathcal A_i(\xi):=A_i\xi+\beta_i$ for $i=1,\ldots,n$. 
Indeed, the span of such {\it deep neurons} is dense in the space of continuous operators on $\mathfrak{X}$, i.e., these deep infinite dimensional neural networks are universal approximants (see \cite[Prop. 5.2]{FrechetNN} and \cite[Prop. 2.10]{FS-paper}). Notice that the universal approximation theorems do not provide any quantification on the number of neurons $M$ nor the depth $n$ in the approximating networks.

\section{A continuity analysis of the nonlinear operator maps}
\label{sect:cont}

In this section, we first define the solution operator for the parabolic Cauchy problem that we aim to solve. To achieve this, we embed the parameters of the Cauchy problem into a suitable Banach space of functions. Next, we analyze the continuity of the operator maps, which act as mappings from a set of parameter functions of interest to solutions. This continuity analysis justifies the use of the Fr\'echet neural network structures introduced in Section \ref{NN:Intro}. Finally, we derive some robustness results for the single-point solution operator.

\subsection{The parabolic Cauchy problem}
\label{subsect:cauchy}

We are going to follow the classical setup provided by  
\cite{friedman1975}. We aim at solving the following backward parabolic Cauchy problem in $\R^n$
\begin{equation}\label{eq: parabolic Cauchy problem}
    \begin{cases}
    Lu + \partial_tu = f(x,t), \quad \text{in } \R^n\times [0,T)\\
    u(x,T) = \phi(x), \quad \text{in } \R^n,
    \end{cases}
\end{equation}
where $0<T<\infty$ and
\begin{equation}\label{eq: elliptic operator}
    Lu = \frac12 \sum_{i,j=1}^n a_{ij}(x,t) \partial^2_{ij}u + \sum_{i=1}^nb_i(x,t)\partial_i u + c(x,t) u.
\end{equation}

We are going to make the following assumptions on the coefficients of $L$, the final datum $\phi$ and the forcing term $f$.
\begin{assumption}
\label{basic-pde-assumption}
The functions $a_{ij}, b_i, c, \phi$ and $f$ satisfy:
\begin{enumerate}
    \item[(i)] There exists a number $\delta>0$ such that $\sum_{i,j=1}^n a_{ij}(x,t)y_iy_j\geq \delta\abs{y}^2$, for any $(x,t)\in \R^n\times [0,T], y \in \R^n$;
    \item[(ii)] $a_{ij}$ and $b_i$ are bounded in $\R^n\times [0,T]$ and Lipschitz continuous in $(x,t)$ in compact subsets of $\R^n\times [0,T]$. The functions $a_{ij}$ are H\"older continuous in $x$, uniformly with respect to $(x,t)\in\R^n\times [0,T]$;
    \item[(iii)] $c$ is bounded in $\R^n\times [0,T]$ and H\"older continuous in $(x,t)$ in compact subsets of $\R^n\times [0,T]$;
    \item[(iv)] $f(x,t)$ is continuous in $\R^n\times [0,T]$, H\"older continuous in $x$ uniformly with respect to $(x,t)\in\R^n\times [0,T]$ and $\abs{f(x,t)}\leq \kappa(1+\abs{x})^{\gamma}$ in $\R^n\times [0,T]$, $\phi(x)$ is continuous in $\R^n$ and $\abs{\phi(x)}\leq \kappa(1+\abs{x})^{\gamma}$, where $\kappa,\gamma$ are positive constants.
\end{enumerate}
\end{assumption}
This set of assumptions ensures that (see \cite[Thm 5.3, p. 148]{friedman1975}) there exists a unique $u\in C(\R^n\times [0,T])\cap C^{1,2}(\R^n\times [0,T))$ solution of \eqref{eq: parabolic Cauchy problem} such that $\abs{u(x,t)}\leq \kappa (1+\abs{x})^{\gamma}$ for some $\kappa>0$. We remark that $\kappa$ is here and throughout a generic constant that is allowed to change according to the context.  
The solution may be represented by the Feynman-Kac formula 
\begin{equation}\label{eq: Feynman-Kac formula}
   \begin{split}
        u(x,t) &= \EE\left[\phi(X_{x,t}(T))\exp\left(\int_t^T c(X_{x,t}(s),s)\,ds\right)\right]\\
        &\quad\quad
        -\EE\left[\int_t^T f(X_{x,t}(s),s) \exp\left(\int_t^s c(X_{x,t}(r),r)\,dr\right) ds\right],
   \end{split}
\end{equation}
for $(x,t)\in \R^n\times [0,T]$, where $t\leq s \leq T$, and
\begin{equation}\label{eq: SDE associated to the parabolic Cauchy problem}
    X_{x,t}(s) = x + \int_t^s b(X_{x,t}(r),r)\,dr+ \int_t^s \eta(X_{x,t}(r),r)\,dW(r) 
\end{equation}
with $\eta\eta^\ast=a$. 
Here, $W$ is an $n$-dimensional Brownian motion defined on a filtered complete probability space $(\Omega,(\mathcal F_t)_{t\in[0,T]},\mathcal F,\mathbb P)$, $\eta\in\R^{n\times n},$ and $\mathbb E[\cdot]$ is the expectation operator with respect to the probability measure $\mathbb P$. Moreover, by \cite[p. 112]{friedman1975}, for any $R>0,0\leq \tau\leq T$, if $\abs{x},\abs{y}\leq R$, there exists $C_{R,T}>0$ such that
\begin{equation}\label{eq: stability estimate for the SDE}
    \EE\left[\sup_{\tau\leq s \leq T} \abs{X_{x,\tau}(s)-X_{y,\tau}(s)}^2\right] \leq C_{R,T} \abs{x-y}^2.
\end{equation}

\subsection{Solution maps and continuity}
Our goal is to learn the nonlinear 
operator map 
\begin{equation}
(\phi,c,f)\stackrel{\ft}{\longmapsto} u(\cdot,t),\quad 0\le t<T
\end{equation}
with the Fr\'echet neural network structures presented in Section \ref{NN:Intro}. In order to do so, we first need a space for $(\phi,c,f)$ and $u(\cdot,t)$ to live in, and then show that the map $(\phi,c,f)\stackrel{\ft}{\longmapsto} u(\cdot,t)$ has the required continuity properties such that we can expect the neural network to approximate the map sufficiently well, in light of the Universal Approximation Theorem \ref{thm:UAT}.

We first introduce a couple of spaces of real valued functions that we will need in the following. We refer the reader to \cite{adams2003sobolev} for these definitions.
For $j\in \N_0$, denote by $C^j (\mathbb{R}^n)$ the space of real valued functions on $\mathbb{R}^n$ whose derivatives up the order $j$ exist. 
Define 
\begin{equation}\label{eq: Cm}
    C^j(\overline{\R^n}) := \{
    v\in C^j(\R^n): \, D^\alpha v \text{ is bounded and uniformly continuous, }\,0\leq \abs{\alpha}\leq j
    \}
\end{equation}
where clearly $D^\alpha v = \partial_{\alpha_1,\dots,\alpha_n}v$ with $\alpha=(\alpha_1,\dots,\alpha_n)\in \N_0^n $. Observe that the notation $C^j(\overline{\R^n})$ arises from the fact that for more general spaces $\mathcal{X}$ instead of $\mathbb{R}^n$, the uniform continuity allows for an extension to the closure $\overline{\mathcal{X}}$. In our case of course $\mathbb{R}^n=\overline{\mathbb{R}^n}$, but $C^j(\overline{\R^n})\neq C^j(\R^n)$.

The vector space $C^j(\overline{\R^n})$ is a Banach space when endowed with the norm given by
$$
\norm{v}_{ C^j(\overline{\R^n})}:=\max_{0\leq\abs{\alpha}\leq j}\sup_{x\in\R^n}\abs{D^\alpha v(x)}.
$$
We further consider for $0<\lambda\leq 1$ also the spaces $C^{j,\lambda}(\overline{\R^n})\subset C^j(\overline{\R^n})$ defined by
\begin{equation}\label{eq: Cml}
    C^{j,\lambda}(\overline{\R^n}) := \{
    v\in C^j(\overline{\R^n}): \, D^\alpha v \text{ satisfies a } \lambda-\text{H\"older condition}\, ,0\leq \abs{\alpha}\leq j
    \}
\end{equation}
i.e. for any $ 0\leq \abs{\alpha}\leq j$ there exists a constant $K\geq 0$ such that
$$
\abs{D^\alpha v(x) - D^\alpha v(y)}\leq K \abs{x-y}^\lambda,\quad x,y\in \R^n.
$$
The space $C^{j,\lambda}(\overline{\R^n})$ is a Banach space when endowed with the norm given by
$$
\norm{v}_{ C^{j,\lambda}(\overline{\R^n})}:=
\norm{v}_{ C^j(\overline{\R^n})}
+
\max_{0\leq\abs{\alpha}\leq j}\sup_{x\neq y}\frac{\abs{D^\alpha v(x) - D^\alpha v(y)}}{\abs{x-y}^\lambda}.
$$
For brevity reasons we write $ C^{j,\lambda}$ rather than $ C^{j,\lambda}(\overline{\R^n})$. We also recall the definition of the Sobolev space $W^{k,p}=W^{k,p}(\R^n),\,k\in\N,1\le p<\infty$: this is the space of all functions $v\in L^p(\R^n)$ such that all their distributional derivatives $D^\alpha v,\,0<\abs{\alpha}\le k$ are functions in $L^p(\R^n)$. A natural norm for these spaces is provided by
\[
\norm{v}_{W^{k,p}}:= \left(
\sum_{0\le\abs{\alpha}\le k} \norm{D^\alpha v}_{L^p}^p
\right)^{1/p}. 
\]
In this way, the spaces become Banach spaces and for $p=2$ Hilbert spaces. These spaces enjoy the following embedding results, known as Sobolev-Morrey embedding: let $j,m\in\N_0$ and
$1\leq p <\infty$.
By part III of Theorem 5.4 and Remark 5.5 (3) \cite[page 98]{adams2003sobolev} the following holds:
\begin{itemize}
    \item If $mp>n>(m-1)p$, then
    \begin{equation}\label{eq: Sob emb 1}
    W^{j+m,p} \hookrightarrow  C^{j,\lambda},\quad 
    0<\lambda \leq m-n/p
    \end{equation}
    with continuity.
     \item If $n = (m-1)p$, then
    \begin{equation}\label{eq: Sob emb 2}
    W^{j+m,p} \hookrightarrow  C^{j,\lambda},\quad 
    0<\lambda <1
    \end{equation}
    with continuity. If $n=m-1$ and $p=1$, then $\lambda$ can be also equal to 1 in the last equation.
\end{itemize}

This result has the following consequence, which we are going to use later: by the continuity of the embedding, we have  
$$
\norm{v}_{C^{j,\lambda}} \leq C_{Sob} \norm{v}_{W^{j+m,p}}, \quad v \in W^{j+m,p}
$$
for some universal constant $C_{Sob}= C_{Sob}(n,j,m,p)$, and so we deduce in particular that
\begin{equation}\label{eq: Hoelder condition j}
    \max_{0\leq\abs{\alpha}\leq j}\sup_{x\neq y}\frac{\abs{D^\alpha v(x) - D^\alpha v(y)}}{\abs{x-y}^\lambda} \leq C_{Sob} \norm{v}_{W^{j+m,p}}.
\end{equation}

From now on, we will assume $j=0$ and that $m$ and $p$ satisfy one of the two conditions above. Thus, $W^{m,p}$ embeds continuously in $C^{0,\lambda}$ for suitable $\lambda$. In this case, \eqref{eq: Hoelder condition j} becomes 
\begin{equation}\label{eq: Hoelder condition}
    \abs{v(x)-v(y)} \leq C_{Sob}\norm{v}_{W^{m,p}} \abs{x-y}^\lambda, \quad x,y\in \R^n, v\in W^{m,p}.
\end{equation}
\begin{remark}
    If we choose $m=1$, then we need $n<p$ and thus $0<\lambda<1-n/p$.
\end{remark}

Going back to \eqref{eq: parabolic Cauchy problem}, we suppose the following to hold: 
\begin{assumption}
For the functions $\phi, c$ and $f$, 
\begin{itemize}
    \item $\phi\in W^{m,p}$
    \item $c$ and $f$ are time-independent and $c,f\in W^{m,p}$. 
\end{itemize}
\end{assumption}
Under this assumption we can conclude from the embeddings that $\phi,c,f\in C^{0,\lambda}$. In view of these embeddings, there exists a unique solution $u$ of \eqref{eq: parabolic Cauchy problem}. 
We fix once for all $0\leq t < T$, and define 
\begin{equation}\label{eq: solution operator (2)}
    \ft: W^{m,p} \times W^{m,p} \times W^{m,p}\to BC(\R^n),\quad (\phi,c,f)\stackrel{\ft}{\longmapsto} u(\cdot,t),
\end{equation}
where $BC(\R^n)$ denotes the space of bounded continuous functions from $\R^n$ to $\R$. Observe that indeed $u\in BC(\R^n\times [0,T])$, because $\phi,c$ and $f$ are bounded and boundedness of $u$ then follows from \eqref{eq: Feynman-Kac formula}.

We have the following continuity result for the operator $F^t$ in \eqref{eq: solution operator (2)} which paves the way for using our neural network in infinite dimension for learning.
\begin{proposition}\label{prop: continuity of the solution operator (2)}
Let $\X:= W^{m,p} \times W^{m,p}\times W^{m,p}$ with $m$ and $p$ satisfying the conditions above. Endow $\X$ with the natural norm
\begin{equation*}
    \norm{(v_1,v_2,v_3)}_\X = 
    \norm{v_1}_{W^{m,p}} + 
     \norm{v_2}_{W^{m,p}} + 
     \norm{v_3}_{W^{m,p}} ,\quad (v_1,v_2,v_3)\in \X.
\end{equation*}
Let $0\leq t < T$. Then, the solution operator $\ft$ defined by \eqref{eq: solution operator (2)} is continuous from $\X$ into $ BC(\R^n)$.
\end{proposition}
\begin{proof}
    Let
\begin{equation*}
    \phi_k \stackrel{W^{m,p}}{\longrightarrow}\phi,\quad c_k \stackrel{W^{m,p}}{\longrightarrow} c,\quad f_k \stackrel{W^{m,p}}{\longrightarrow} f 
\end{equation*}
as $k\to \infty$, and let $u_k$ be the corresponding solution of \eqref{eq: parabolic Cauchy problem}.
We observe the following elementary facts:
\begin{enumerate}
    \item Since $\phi_k\to \phi$ uniformly on $\R^n$ by the embedding into $C^{0,\lambda}$, we have $\phi_k(X_{x,t}(T))\to\phi(X_{x,t}(T))$ uniformly in $(\omega,x)\in\Omega\times\R^n$ as $k\to\infty$.
    \item Similarly, since $c_k\to c$ uniformly on $\R^n$, we have $c_k(X_{x,t}(s))\to c(X_{x,t}(s))$ uniformly in $(\omega,x,s)\in\Omega\times\R^n\times [t,T]$ as $k\to\infty$. Therefore,
    \begin{equation*}
        \int_t^T c_k(X_{x,t}(s))\,ds\to \int_t^T  c(X_{x,t}(s))\,ds
    \end{equation*}
    uniformly in $(\omega,x)\in\Omega\times\R^n$.
    \item Using the mean value theorem and the fact that the quantities $ \int_t^T c_k(X_{x,t}(s))\,ds$ and $ \int_t^T c(X_{x,t}(s))\,ds$ are bounded, it is immediate to see that 
    \begin{equation*}
        \exp\left\{\int_t^T c_k(X_{x,t}(s))\,ds\right\}\to \exp\left\{\int_t^T  c(X_{x,t}(s))\,ds\right\}
    \end{equation*}
    uniformly in $(\omega,x)\in\Omega\times\R^n$ as $k\to\infty$.
    \item Mutatis mutandis, $\int_t^s c_k(X_{x,t}(r))\,dr\to \int_t^s  c(X_{x,t}(r))\,dr$ uniformly in $(\omega,x,s)\in\Omega\times\R^n\times [t,T]$ and once more
    \begin{equation*}
        \exp\left\{\int_t^s c_k(X_{x,t}(r))\,dr\right\}\to \exp\left\{\int_t^s  c(X_{x,t}(r))\,dr\right\}
    \end{equation*}
     uniformly in $(\omega,x,s)\in\Omega\times\R^n\times [t,T]$ as $k\to\infty$.
    \item $f_k\to f$ uniformly on $\R^n$ implies $f_k(X_{x,t}(s))\to f(X_{x,t}(s))$ uniformly in $(\omega,x,s)\in\Omega\times\R^n\times [t,T]$. 
\end{enumerate}
In view of this, we deduce that $\sup_{\R^n}\abs{u_k(x,t)-u(x,t)}\to 0$ as $k\to\infty$. We have proved the proposition. 
\end{proof}
We observe that the space $\X$ defined in Proposition \ref{prop: continuity of the solution operator (2)} above is a separable Banach space, under the natural linear structure of product of spaces. It can be easily supplemented by a Schauder basis, coming from a Schauder basis of $W^{m,p}$ (see e.g. \cite[Ch. 4]{Heil} for a general introduction to basis functions in Banach spaces). Indeed, from \cite{Triebel} we know that these Sobolev spaces $W^{m,p}$ carry an unconditional Schauder basis given by wavelets. We refer to \cite[Ch. 12]{Heil} and \cite{Meyer} for more on wavelets. In Section \ref{sect:numerics} we present a numerical example where we construct a basis instead in terms of Hermite functions for the Hilbert space $W^{1,2}$.

We finally observe that for arbitrary $z\in\R^n$, the Dirac mass $\delta_z$ is trivially an element of the topological dual of $BC(\R^n)$. In view of all of this, we conclude that the map $\langle \delta_z,\ft\rangle$ is an element of $C(\X)$. Let us show how we can utilize this to learn the solution map uniformly on small compact subsets of $\R^n$ (see also Corollary \ref{corr: to prop uniform learning of u on small compact subsets (2)} below).

\begin{proposition}\label{prop: uniform learning of u on small compact subsets (2)}

Fix $0\leq t <T$, $\mathcal K\subset\X$ compact and $R>0$. Let $\lambda$ be the 
H\"older constant provided by equation \eqref{eq: Hoelder condition}. Then there exists a constant $\Gamma=\Gamma(\mathcal K,T,R)>0$ such that
\begin{equation*}
    \abs{\langle \delta_x-\delta_y,\ft(\phi,c,f)\rangle} \leq \Gamma(\mathcal K,T,R) \abs{x-y}^\lambda
\end{equation*}
for $(\phi,c,f)\in\mathcal K$ and $\abs{x}\leq R, \abs{y}\leq R$.
\end{proposition}
\begin{proof}
In the following computations we are going to use repeatedly the H\"older condition in \eqref{eq: Hoelder condition}
as well as
$$
\norm{v}_\infty \leq C_{Sob}\norm{v}_{W^{m,p}}.
$$
We define for convenience the map $I$,
\begin{equation*}
    I : W^{m,p}\to \R, \quad v \mapsto
      \exp\{(T-t)\norm{v}_{\infty}\}.
\end{equation*}
We observe that it is a continuous map on $W^{m,p}$, because of the continuity of the embedding and by composition of continuous maps, namely
$$
W^{m,p}\ni v \mapsto v \in C^{0,\lambda} \mapsto \norm{v}_\infty \mapsto I(v).
$$
Let $x,y\in\R^n$, and set $u(\cdot,t)=\ft(\phi,c,f)$. Then, using Feynman-Kac formula \eqref{eq: Feynman-Kac formula}, and several times the mean value theorem for the terms involving the exponential function, we obtain 
\begin{equation*}
    \begin{split}
       &\abs{u(x,t)-u(y,t)} \leq \EE\abs{\phi(X_{x,t}(T))-\phi(X_{y,t}(T))}\exp\left[\int_t^T c(X_{x,t}(s))\,ds\right]\\
       &\quad\quad\quad + \EE\abs{\phi(X_{y,t}(T))}\abs{\exp\left[\int_t^T c(X_{x,t}(s))\,ds\right]-\exp\left[\int_t^T c(X_{y,t}(s))\,ds\right]}\\
       &\quad\quad\quad + \EE\int_t^T\abs{f(X_{x,t}(s))-f(X_{y,t}(s))}\exp\left[\int_t^s c(X_{x,t}(r))\,dr\right]ds\\
       &\quad\quad\quad + \EE\int_t^T\abs{f(X_{y,t}(s))}\abs{\exp\left[\int_t^s c(X_{x,t}(r))\,dr\right]-\exp\left[\int_t^s c(X_{y,t}(r))\,dr\right]}ds\\
       &\quad\quad \leq C_{Sob} \EE \norm{\phi}_{W^{m,p}} I(c)\abs{X_{x,t}(T)-X_{y,t}(T)}^\lambda\\
       &\quad\quad\quad +\EE \norm{\phi}_\infty I(c)\abs{\int_t^T[c(X_{x,t}(s))-c(X_{y,t}(s))]ds}\\
       &\quad\quad\quad +C_{Sob}\EE\int_t^T \norm{f}_{W^{m,p}}I(c)\abs{X_{x,t}(s)-X_{y,t}(s)}^\lambda ds\\
       &\quad\quad\quad +\EE\int_t^T \norm{f}_\infty  I(c)\abs{\int_t^s[c(X_{x,t}(r))-c(X_{y,t}(r))]dr}ds.
    \end{split}
\end{equation*}
Thus, by applying H\"older's inequality repeatedly ($2/\lambda \ge 1$),
\begin{equation*}
    \begin{split}
        &\abs{u(x,t)-u(y,t)} \leq
    C_{Sob} \norm{\phi}_{W^{m,p}} I(c)\left[\EE\abs{X_{x,t}(T)-X_{y,t}(T)}^2\right]^{\lambda/2}\\
       &\quad\quad\quad + \norm{\phi}_{\infty}I(c) C_{Sob}\norm{c}_{W^{m,p}} \EE\int_t^T\abs{X_{x,t}(s)-X_{y,t}(s)}^\lambda ds\\
       &\quad\quad\quad + C_{Sob}\norm{f}_{W^{m,p}} I(c)\EE\int_t^T \abs{X_{x,t}(s)-X_{y,t}(s)}^\lambda ds\\
       &\quad\quad\quad + C_{Sob}\norm{f}_\infty I(c) \norm{c}_{W^{m,p}}
       \EE\int_t^T\int_t^s  \abs{X_{x,t}(r)-X_{y,t}(r)}^\lambda drds\\
       &\quad\quad \leq C_{Sob} \norm{\phi}_{W^{m,p}} I(c)\left[\EE\abs{X_{x,t}(T)-X_{y,t}(T)}^2\right]^{\lambda/2}\\
       &\quad\quad\quad + \norm{\phi}_{W^{m,p}}I(c) C^2_{Sob}\norm{c}_{W^{m,p}} \left[\EE\int_t^T\abs{X_{x,t}(s)-X_{y,t}(s)}^2 ds\right]^{\lambda/2}(T-t)^{1-\lambda/2}\\
       &\quad\quad\quad + C_{Sob}\norm{f}_{W^{m,p}} I(c)\left[\EE\int_t^T\abs{X_{x,t}(s)-X_{y,t}(s)}^2 ds\right]^{\lambda/2}(T-t)^{1-\lambda/2}\\
       &\quad\quad\quad + C^2_{Sob}\norm{f}_{W^{m,p}} I(c) \norm{c}_{W^{m,p}}
       \left[\EE\int_t^T\int_t^s  \abs{X_{x,t}(r)-X_{y,t}(r)}^2 drds\right]^{\lambda/2}\left[\frac{(T-t)^2}{2}\right]^{1-\lambda/2}.
    \end{split}
\end{equation*}

Let $x,y:\abs{x}\leq R,\abs{y}\leq R$ for some fixed $R>0$. Using \eqref{eq: stability estimate for the SDE} we infer
\begin{equation*}
    \begin{split}
        \abs{u(x,t)-u(y,t)} &\leq  C_{Sob} \norm{\phi}_{W^{m,p}}I(c) C_{R,T}^{\lambda/2}\abs{x-y}^\lambda \\
        &\quad + C^2_{Sob} \norm{\phi}_{W^{m,p}}
        I(c) \norm{c}_{W^{m,p}} C_{R,T}^{\lambda/2}(T-t)\abs{x-y}^\lambda \\
        & \quad+  C_{Sob} \norm{f}_{W^{m,p}}I(c) C_{R,T}^{\lambda/2} (T-t)\abs{x-y}^\lambda \\
        &\quad + C^2_{Sob} \norm{f}_{W^{m,p}} I(c) \norm{c}_{W^{m,p}}  C_{R,T}^{\lambda/2} \frac{(T-t)^2}{2}\abs{x-y}^\lambda. 
    \end{split}
\end{equation*}
Therefore, by continuity and Weierstrass theorem, we conclude that there exists a constant $\Gamma=\Gamma(\mathcal K,T,R)>0$ such that 
\begin{equation*}
    \abs{u(x,t)-u(y,t)} \leq \Gamma(\mathcal K,T,R) \abs{x-y}^\lambda,\quad (\phi,c,f)\in \mathcal K, \abs{x}\leq R, \abs{y}\leq R, 
\end{equation*}
i.e. 
\begin{equation*}
    \abs{\ft(\phi,c,f)(x)-\ft(\phi,c,f)(y)} \leq \Gamma(\mathcal K,T,R) \abs{x-y}^\lambda,\quad (\phi,c,f)\in\mathcal K, \abs{x}\leq R, \abs{y}\leq R.
\end{equation*}
We have proved the proposition. 
\end{proof}

\begin{remark}
    We remark that the result in 
    Proposition \ref{prop: uniform learning of u on small compact subsets (2)} would hold for subsets $\mathcal K\subset \X$ being bounded only and not necessarily being compact. However, since the Universal Approximation Theorem \ref{thm:UAT} in any case requires working on compact sets, we have formulated the proposition accordingly.  
\end{remark}

We have the following important consequence, which in broad strokes tells us that, as long as we stay close to $x$ and willing to accept a slightly higher error, we do not need to change the approximating neural network architecture.
\begin{corollary}\label{corr: to prop uniform learning of u on small compact subsets (2)}
Fix $0\leq t <T$, $\mathcal K\subset\X$ compact and $R>0$. Let $x:\abs{x}\leq R$. Let $\varepsilon>0$ be arbitrary. Suppose to be given for some $N\in\mathbb N$ 
\begin{equation*}
    \mathcal{N}^N = \sum_{j=1}^N  \mathcal{N}_{\ell_j,A_j,\beta_j}
\end{equation*}
with $\ell_j\in\X',A_j\in\mathcal{L}(\X)$ and $\beta_j\in\X$ such that 
\begin{equation*}
\sup_{(\phi,c,f)\in \mathcal K} \abs{\mathcal{N}^N(\phi,c,f) - \langle \delta_x,\ft(\phi,c,f)\rangle} <\varepsilon.
\end{equation*}
Fix $\varepsilon'>0$ and set $r=\left(\frac{\varepsilon'}{\Gamma(\mathcal K,T,R)}\right)^{1/\lambda}$. Then for any $y\in B_r(x),\abs{y}\leq R$, it holds 
\begin{equation*}
\sup_{(\phi,c,f)\in K} \abs{(\mathcal{N}^N(\phi,c,f) - \langle \delta_y,\ft(\phi,c,f)\rangle} <\varepsilon + \varepsilon'.
\end{equation*}
\end{corollary}
\begin{proof}
From the previous proposition, we indeed have for any $y\in B_r(x),\abs{y}\leq R$
\begin{equation*}
     \abs{\langle \delta_x-\delta_y,\ft(\phi,c,f)\rangle} < \varepsilon'
\end{equation*}
for any $(\phi,c,f)\in\mathcal K$, and therefore
\begin{equation*}
    \sup_{(\phi,c,f)\in \mathcal K}\abs{\langle \delta_x,\ft(\phi,c,f)\rangle - \langle \delta_y,\ft(\phi,c,f)\rangle} < \varepsilon'.
\end{equation*}
By the triangle inequality we get the claim.
\end{proof}

Needless to say, everything said until here still holds for possibly different solution operators $F^t$ where some of the ``variables''defining the parabolic Cauchy problem are fixed, i.e., given exogenously. However, by relaxing the properties of these fixed variables, we can allow for more flexible specifications but still preserve the continuity of the solution operator. We investigate this next in the case where $f=0$ and the final datum $\phi$ is not necessarily in $W^{m,p}$ but is continuous with some  polynomial growth (i.e., still satisfying the standard assumptions ensuring well-posedness of \eqref{eq: parabolic Cauchy problem}). 
Namely, we consider solutions of this kind:
\begin{equation*}
        u(x,t) = \EE\left[\phi(X_{x,t}(T))\exp\left(\int_t^T c(X_{x,t}(s))\,ds\right)\right],\quad (x,t)\in \R^n\times [0,T]
\end{equation*}
with $c\in W^{m,p}$ and $\phi$ continuous in $\R^n$ such that $\abs{\phi(x)}\leq \kappa(1+\abs{x})^{\gamma}$, where $\kappa,\gamma$ are positive constants. From the general theory in Subsection \ref{subsect:cauchy} we know that $\abs{u(x,t)}\leq \kappa (1+\abs{x})^{\gamma}$ for some $\kappa>0$ (possibly different to the above), and thus in general the solution will be unbounded. To overcome this issue, we simply restrict ourselves to a fixed compact subset $K\subset \R^n$, and we will learn the solution here. More precisely, for $0\leq t < T$, $K\subset \R^n$ compact and $\phi$ as above we define the following solution operator

\begin{equation}\label{eq: solution operator 2 (2)}
    \fth: W^{m,p}\to C(K),\quad c\longmapsto u(\cdot,t)\Big|_K
\end{equation}
namely the solution of \eqref{eq: parabolic Cauchy problem} as a function of $c$ restricted to $K$ with $f=0$, final datum $\phi$ given. Also in this setting we obtain:

\begin{proposition}\label{prop: continuity of the solution operator 2 (2)}
Let $\phi$ be continuous in $\R^n$ and such that $\abs{\phi(x)}\leq \kappa(1+\abs{x})^{\gamma}$, where $\kappa,\gamma$ are positive constants. Assume $c\in W^{m,p}$, $0\leq t <T$ and let $K\subset\R^n$ be compact. Then the solution operator $\fth$ of the parabolic Cauchy problem
\begin{equation*}
    \begin{cases}
    Lu + \partial_tu  =0, \quad \text{in } \R^n\times [0,T)\\
    u(x,T) = \phi(x), \quad \text{in } \R^n,
    \end{cases}
\end{equation*}
is continuous from $W^{m,p}$ into $C(K)$.

If $\phi$ is assumed additionally to be H\"older continuous with exponent $0<\tilde{\lambda}\leq 1$, the following holds: given a compact subset $\mathcal K\subset W^{m,p}$, there is $\Gamma=\Gamma(\phi,\tilde{\lambda},T,a_{ij},b_i,K,\mathcal K)>0$ constant such that 
\begin{equation*}
    \abs{
    \langle \delta_x - \delta_y, \fth(c)\rangle} \leq \Gamma \abs{x-y}^{\tilde{\lambda}},\quad c\in \mathcal K, \,x,y\in K.
\end{equation*}
\end{proposition}

\begin{proof}

Let us show continuity first. Given $c_k\stackrel{W^{m,p}}{\longrightarrow}c$, we know from before that $$ \exp\left\{\int_t^T c_k(X_{x,t}(s))\,ds\right\}\to \exp\left\{\int_t^T  c(X_{x,t}(s))\,ds\right\}$$ uniformly on $\Omega\times K$. Furthermore, from standard SDEs theory (see for instance \cite[Thm 2.3 page 107]{friedman1975}), we know that for any $h\in\N$
\begin{equation*}
    \EE\abs{X_{x,t}(s)}^h \leq (2 + \abs{x}^h)e^{Cs},\quad t\leq s \leq T, \,x\in\R^n
\end{equation*}
where $C=C(h,a_{ij},b_i,T)$. Thus, in view of the bound satisfied by $\phi$, we easily get
\begin{equation*}
    \abs{\phi(X_{x,t}(T))} \lesssim_\phi (1 + \abs{X_{x,t}(T)}^h),\quad x\in K 
\end{equation*}
for some $h=h(\phi)\in\N$, and hence
\begin{equation*}
    \EE \abs{\phi(X_{x,t}(T))} \leq C(\phi,T,a_{ij},b_i,K)
\end{equation*}
uniformly in $x\in K$. From this we infer
\begin{equation*}
    \begin{split}
        &\abs{u_k(x,t) -u(x,t)}  \leq  \EE \abs{\phi(X_{x,t}(T))} \abs{\exp\left[\int_t^T c_k(X_{x,t}(s))\,ds\right]-\exp\left[\int_t^T c(X_{x,t}(s))\,ds\right]}\\
        & \quad\quad \leq \EE \abs{\phi(X_{x,t}(T))} \sup_{(\omega,x)\in\Omega\times K}\abs{\exp\left[\int_t^T c_k(X_{x,t}(s))\,ds\right]-\exp\left[\int_t^T c(X_{x,t}(s))\,ds\right]}\\
        & \quad\quad \leq C(\phi,T,a_{ij},b_i,K) \sup_{(\omega,x)\in\Omega\times K}\abs{\exp\left[\int_t^T c_m(X_{x,t}(s))\,ds\right]-\exp\left[\int_t^T c(X_{x,t}(s))\,ds\right]}
    \end{split}
\end{equation*}
uniformly in $x\in K$. Therefore, $\sup_{x\in K}\abs{u_k(x,t) -u(x,t)}\to 0$ as $k\to\infty$, namely $\fth$ is continuous on $W^{m,p}$ into $C(K)$.

Let us now assume additionally that $\phi$ is H\"older continuous for some exponent $0<\tilde{\lambda}\leq 1$. Arguing as above in the proof of Proposition \ref{prop: uniform learning of u on small compact subsets (2)}, for $x,y\in K$ we now have
\begin{equation*}
    \begin{split}
       &\abs{u(x,t)-u(y,t)} \leq \EE\abs{\phi(X_{x,t}(T))-\phi(X_{y,t}(T))}\exp\left[\int_t^T c(X_{x,t}(s))\,ds\right]\\
       &\quad\quad + \EE\abs{\phi(X_{y,t}(T))}\abs{\exp\left[\int_t^T c(X_{x,t}(s))\,ds\right]-\exp\left[\int_t^T c(X_{y,t}(s))\,ds\right]}\\
       &\quad\quad \leq 
       C_\phi
       I(c)\EE\abs{X_{x,t}(T)-X_{y,t}(T)}^{\tilde{\lambda}}\\
       &\quad\quad\quad + I(c) \EE \abs{\phi(X_{y,t}(T))}
       \abs{\int_t^T[c(X_{x,t}(s))-c(X_{y,t}(s))]ds}\\
       &\quad\quad \leq C_\phi I(c)\left[\EE\abs{X_{x,t}(T)-X_{y,t}(T)}^2\right]^{\tilde{\lambda}/2}\\
       &\quad\quad\quad + I(c) C_{Sob}\norm{c}_{W^{m,p}}
       \EE \abs{\phi(X_{y,t}(T))}
       \int_t^T\abs{X_{x,t}(s)-X_{y,t}(s)}^{\tilde{\lambda}} ds.
    \end{split}
\end{equation*}
By H\"older's and Jensen's inequalities we infer
\begin{equation*}
    \begin{split}
         \EE \abs{\phi(X_{y,t}(T))} &
       \int_t^T\abs{X_{x,t}(s)-X_{y,t}(s)}^{\tilde{ \lambda}} ds  \leq
       \left[
        \EE \abs{\phi(X_{y,t}(T))}^{2/(2-\tilde{\lambda})}
       \right]^{1-\tilde{\lambda}/2}\;\times \\
       & \quad \times  \left[\EE \left[\int_t^T\abs{X_{x,t}(s)-X_{y,t}(s)}^{\tilde{\lambda}} ds \right]^{2/\tilde{ \lambda}}\right]^{\tilde{\lambda}/2}\\
       &\leq  \left[
        \EE \abs{\phi(X_{y,t}(T))}^{2/(2-\tilde{\lambda})}
       \right]^{1-\tilde{\lambda}/2} \; \times \\
       & \quad \times 
       (T-t)^{1-\tilde{\lambda}/2}  \left[\EE \int_t^T\abs{X_{x,t}(s)-X_{y,t}(s)}^2 ds \right]^{\tilde{\lambda}/2}.
    \end{split}
\end{equation*}

From above, we deduce
\begin{equation*}
    \abs{\phi(X_{y,t}(T))}^{2/(2-\tilde{\lambda})} \lesssim_{\phi,\tilde{\lambda}} (1 + \abs{X_{y,t}(T)}^{2h/(2-\tilde{\lambda})}),\quad y\in K 
\end{equation*}
for some $h=h(\phi)\in\N$, and hence, with a different constant,
\begin{equation*}
    \EE \abs{\phi(X_{y,t}(T))}^{2/(2-\tilde{\lambda})} \leq C(\phi,\tilde{\lambda}, T,a_{ij},b_i,K)
\end{equation*}
uniformly in $y\in K$. Similarly to above, we then obtain
\begin{equation*}
    \begin{split}
        \abs{u(x,t)-u(y,t)} &\leq  C_\phi I(c) C_{K,T}^{\tilde{\lambda}/2}\abs{x-y}^{\tilde{\lambda}} +\\ &\quad + C_{Sob}I(c) \norm{c}_{W^{m,p}} C(\phi,\tilde{\lambda}, T,a_{ij},b_i,K) C_{H,T}^{\tilde{\lambda}/2}(T-t)\abs{x-y}^{\tilde{\lambda}},
    \end{split}
\end{equation*}
for $x,y\in K$.
Let $\mathcal K\subset W^{m,p}$ be a fixed compact subset. By continuity and the Weierstrass Theorem, we conclude that there exists a constant $\Gamma=\Gamma(\phi,\tilde{\lambda},T,a_{ij},b_i,K,\mathcal{K})>0$ such that 
\begin{equation*}
    \abs{u(x,t)-u(y,t)} \leq \Gamma \abs{x-y}^{\tilde{\lambda}},\quad c\in\mathcal K, \,x,y\in K, 
\end{equation*}
i.e., 
\begin{equation*}
    \abs{\fth(c)(x)-\fth(c)(y)} \leq \Gamma \abs{x-y}^{\tilde{\lambda}},\quad c\in\mathcal K, \,x,y\in K,
\end{equation*}
and the claim follows.
\end{proof}

As an immediate consequence, we effortlessly obtain the analogous of Corollary \ref{corr: to prop uniform learning of u on small compact subsets (2)}:
\begin{corollary}\label{corr: to prop continuity of the solution operator 2 (2)}
Assume the setting of Proposition \ref{prop: continuity of the solution operator 2 (2)}. Let $x\in K$ and $\varepsilon>0$ be arbitrary. Suppose for $N\in\mathbb N$ to be given 
\begin{equation*}
    \mathcal{N}^N = \sum_{j=1}^N \mathcal{N}_{\ell_j,A_j,\beta_j}
\end{equation*}
with $\ell_j\in (W^{m,p})'$, $A_j\in\mathcal{L}(W^{m,p})$ and $\beta_j\in W^{m,p}$  
such that 
\begin{equation*}
\sup_{c\in\mathcal K}\abs{ \mathcal{N}^N(c) - \langle \delta_x,\fth(c)\rangle} <\varepsilon.
\end{equation*}
Fix $\varepsilon'>0$ and set $r=\left(\frac{\varepsilon'}{\Gamma}\right)^{1/\tilde{\lambda}}$. Then for any $y\in K\cap B_r(x)$, it holds 
\begin{equation*}
\sup_{c\in \mathcal K}\abs{(\mathcal{N}^N(c) - \langle \delta_y,\fth(c)\rangle} <\varepsilon + \varepsilon'.
\end{equation*}
\end{corollary}
After these theoretical considerations on continuity, verifying the use of the Universal Approximation Theorem for operator-learning, we proceed in the next section with a numerical case study.
\section{A numerical case study}
\label{sect:numerics}

We demonstrate our proposed methodology by considering a particular case of the Cauchy problem set on $\mathbb R$ and train a neural network to learn the operator mapping the function $c$ into the $u$ solution evaluated in a location. In our proof-of-concept study, we benchmark with respect to the DeepONet approach.

For our purposes, we need to have a set of orthonormal basis functions in $W^{1,2}$. These provide us with structural information that we exploit in the training. Here we propose to construct such a basis from the Hermite functions, which is an orthonormal basis of $L^2(\mathbb R)$.

\subsection{Basis functions}
Following e.g. \cite[p. 261]{schwartz1950theorie}, we define the 1-d Hermite polynomials by
\begin{equation}\label{eq: Hermite polynomials}
    H_m(x)=(-1)^m2^{1/4-m}(m!)^{-1/2}\pi^{-m/2}\exp(2\pi x^2)\frac{d^m}{dx^m}\exp(-2\pi x^2) 
\end{equation}
and the associated Hermite functions
\begin{equation}\label{eq: Hermite functions 1-d}
    \mathcal{H}_m(x)= H_m(x)\exp(-\pi x^2). 
\end{equation}
where $x\in\R, \,m\in \N_0$.
Then $(\mathcal{H}_m)_{m\in\N_0}$ is an orthonormal system in $L^2(\R)$. We first derive $\langle\mathcal{H}_m,\mathcal{H}_n\rangle_{W^{1,2}}$ which we need in order to obtain an orthonormal set of vectors in $W^{1,2}$.

\begin{proposition}
The following holds: 
\begin{enumerate}
    \item for $m,n\in\N$ 
\[
\begin{split}
    \int_{\mathbb R} \mathcal{H}'_m(x)\mathcal{H}'_n(x)\,dx
    &=\pi(2m+1)\delta_{m,n} -\pi\sqrt{m(m-1)}\,\delta_{m,n+2} \\
    &\qquad\qquad-\pi\sqrt{(m+1)(m+2)}\,\delta_{m,n-2},
\end{split}
\]
\item for $m\in\N$ and $n=0$
\[
\begin{split}
    &\int_{\mathbb R} \mathcal{H}'_m(x)\mathcal{H}'_0(x)\,dx
=-\pi\sqrt{2}\,\delta_{m,2}, 
\end{split}
\]
\item and for $m=n=0$
\[
\int_{\mathbb R} \mathcal{H}'_0(x)\mathcal{H}'_0(x)\,dx
= \pi.
\]
\end{enumerate}
\end{proposition}
\begin{proof}
It follows from \cite[VII, 7; 30]{schwartz1950theorie} (or by direct computation) that
$$
    -\mathcal{H}'_m(x) + 2\pi x\mathcal{H}_m(x) = 2\sqrt{\pi(m+1)}\mathcal{H}_{m+1}(x),\quad m\in \N_0
$$
and
$$
    -\mathcal{H}'_m(x) - 2\pi x\mathcal{H}_m(x) = -2\sqrt{\pi m}\mathcal{H}_{m-1}(x),\quad m\in\N.
$$
By summing up these two equations, one obtains the recursion
$$
    \mathcal{H}'_m(x) = \sqrt{\pi m}\mathcal{H}_{m-1}(x) - \sqrt{\pi(m+1)}\mathcal{H}_{m+1}(x),\quad m\in\N.
$$
Hence, after appealing to the orthogonality of $(\mathcal H_m)_{m\in\mathbb N_0}$, we get
\begin{align*}
&\int_{\mathbb R} \mathcal{H}'_m(x)\mathcal{H}'_n(x)\,dx\\ 
&\qquad=\pi\left(\sqrt{mn} + \sqrt{(m+1)(n+1)}\right)\,\delta_{m,n} - \pi \sqrt{m(n+1)} \,\delta_{m,n+2} - \pi\sqrt{(m+1)n}\,\delta_{m,n-2}
\end{align*}
which proves (1).
For $m=0$ one obtains directly from the definition of $\mathcal{H}_0$ that
$$
    \mathcal{H}'_0(x) = -2\pi x \mathcal{H}_0(x).
$$
From the definition of $H_0$ and $\mathcal{H}_0$ it follows that $H_0(x) = 2^{1/4}$ and $H_1(x)=2^{5/4}\sqrt{\pi}x$.  
Now we can re-write $\mathcal{H}'_0(x)=-2\pi x \mathcal{H}_0(x)$ as
\[
\mathcal{H}'_0(x)=-2\pi x 2^{1/4}\exp(-\pi x^2) = -\sqrt{\pi}\mathcal{H}_1(x)
\]
because $\mathcal{H}_1(x)=2^{1/4}2\sqrt{\pi}x\exp(-\pi x^2)$. With this last observation the case (2) with $m=1$ and case (3) follows.
\end{proof}

To this end, we write the conclusions of the last proposition in a more concise form, resulting in the following ``multiplication table'': for $m,n\in\N_0$ it holds:
\begin{itemize}
    \item $\langle\mathcal{H}_m,\mathcal{H}_n\rangle_{W^{1,2}} = 1+\pi(2m+1) $ for
    $m=n$;
    \item $\langle\mathcal{H}_m,\mathcal{H}_n\rangle_{W^{1,2}}= -\pi \sqrt{\ell(\ell-1)}$ for $\abs{m-n}=2$ with $\ell=\max\{m,n\}$; 
    \item $\langle\mathcal{H}_m,\mathcal{H}_n\rangle_{W^{1,2}} = 0$ otherwise.
\end{itemize}
With this and the fact that 
\[
\int_{\mathbb R} \mathcal{H}_m(x)\mathcal{H}_n(x)\,dx =\delta_{m,n}\,,
\]
we can now apply the Gram–Schmidt procedure to the vectors $\mathcal{H}_0, \mathcal{H}_1, \dots $ to obtain an orthonormal basis in $W^{1,2}$, which we denote by $(e_k)_{k\in\mathbb N_0}$.
We remark in passing that we can build basis functions in $W^{1,2}$ for $d>1$ by tensorising the above Hermite basis.

\subsection{Numerical example}

We consider the parabolic Cauchy problem in $\R$
\begin{equation}
    \begin{cases}
    Lu + \partial_tu = 0, \quad \text{in } \R \times [0,T)\\
    u(x,T) = \phi (x) , \quad \text{in } \R,
    \end{cases}
\end{equation}
where $0<T<\infty$ and
\begin{equation}
    Lu = \frac12  \partial^2_x u  + c(x) u.
\end{equation}
To simplify, we have set $a=1$, $b=0$ and the forcing term $f=0$ in this numerical example. For the final datum $\phi$, we choose $\phi (x)= x^2$. With these specifications, we aim for learning the non-linear operator mapping
\begin{equation}
    c \mapsto u (\cdot, t).
\end{equation}
where, by \eqref{eq: Feynman-Kac formula} and \eqref{eq: SDE associated to the parabolic Cauchy problem} $u$ is given by 
\begin{equation}
        u(x,t) = \EE \left[ \phi(X_{x,t}(T))\exp\left(\int_t^T c(X_{x,t}(s))\,ds\right)\right]
\end{equation}
and 
\begin{equation}\label{X:x:t:s}
    X_{x,t}(s) = x + \int_t^s \,dW(r) .
\end{equation}
For fixed $x\in \mathbb{R}$, we fit a neural network as introduced in Section \ref{NN:Intro} to learn the map $c \mapsto u (x,t)$ on a compact subset $\mathcal K \subset W^{m,p}$. In our numerics, we let $t=0$ and $T=1$.

Instead of using data points $\{ (c_i,u_i (x,t))\}_{i=1}^{ M_{\text{train}}}$ to train the neural network, we instead fit a neural network to minimize the energy functional
\begin{equation}\label{energy}
g\mapsto \mathbb E\left[\int_{W^{1,2}}\vert \mathcal X(c)-g(c)\vert^2\mu(dc)\right],
\end{equation}
with $\mathcal X(c)$ given by 
\begin{equation}
\mathcal X(c):=\phi(X_{x,t}(T))\exp\left(\int_0^1 c(X_{x,t}(s))\,ds\right), 
\end{equation}
and where $\mu$ is a measure on $\mathcal K$. This approach allows us to appeal to the Uniform Approximation Theorem since we have continuity of the map $c \mapsto u (x,t)$. We refer to \cite[Prop. 2.2]{Beck2021} where this approach has been used first in the finite dimensional case, and
\cite[Lem. 5.4]{FS-paper} for its extension to infinite dimensional spaces. 

Recall from above the basis $(e_k)_{k\in\mathbb N_0}$ of Hermite functions for $W^{1,2}$. We choose the compact set $\mathcal K\subset W^{1,2}$ by
$$\mathcal K:=\{c\in W^{1,2}; c=\sum_{k=1}^{5} a_i e_i, a_i\in[-5,5] \}.$$
We specify a uniform measure $\mu$ on $\mathcal K$ canonically from the classical uniform measure on $[-5,5]^5$. and we trivially extend it to the whole space. In Figure \ref{figure:basis:and:data} we show the first $5$ basis vectors and 5 random samples from $\mathcal K$ (Note that the multiple occurrences of 5 is a coincident and not intentional).
\begin{figure}[t]
    \hfill \subfigure[]{\includegraphics[width=0.45\textwidth]{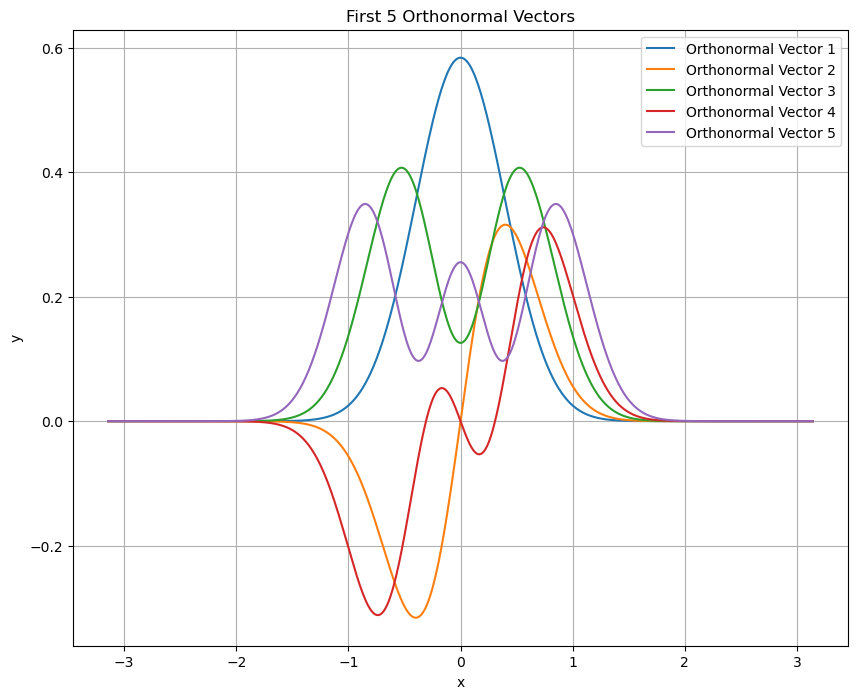}\label{figure:basis}} \hfill\subfigure[]{\includegraphics[width=0.45\textwidth]{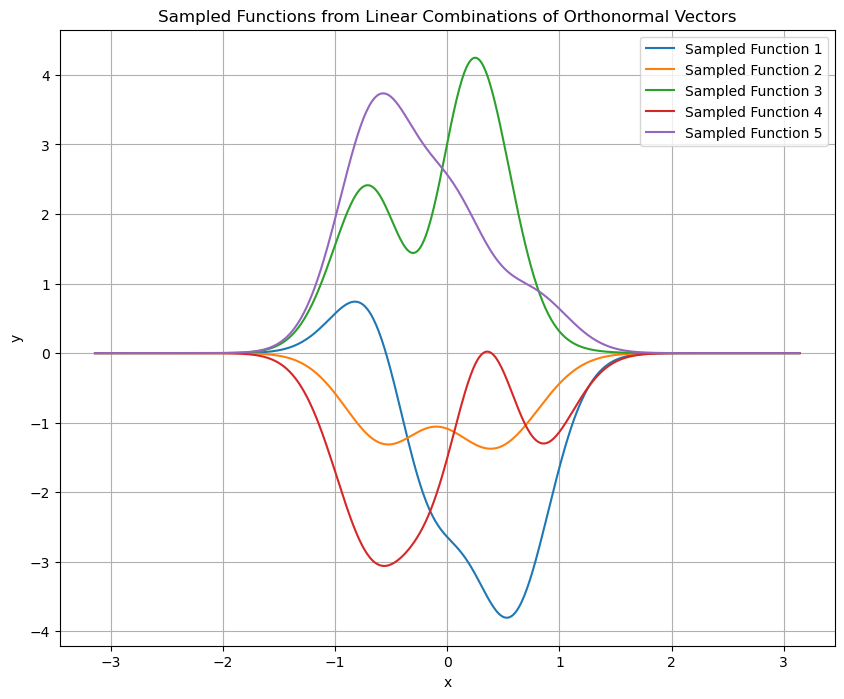}\label{figure:data}} 
\caption{First 5 basis functions $e_1,\ldots,e_5$ (left) and 5 random samples of $c$ from $\mathcal K$ (right)}\label{figure:basis:and:data}
\end{figure}
We fit a Fr\'echet neural network as introduced in Section \ref{NN:Intro} with two layers and $15$ nodes in each layer. For the activation function $\sigma$ we follow Example 4.4 in Benth, Detering and Galimberti \cite{FrechetNN} and specify $\sigma (x)=\beta(\psi (x)) z$ for a $\psi \in \mathcal{L} ( W^{1,2}, \mathbb{R})$, a vector $z\in W^{1,2}$ and $\beta$ is a real-valued Lipschitz continuous function on the real line. 
In particular, we choose   $\beta(y):= \max \{0,1-\exp (-y) \}$,  $\psi(h)=a_1\cdot 0.25 +\dots a_5\cdot 0.25$ for $h=\sum_{k=1}^{\infty} a_i e_i$ and $z=e_1 + \dots + e_5$.

We build a training set of $M_{\text{train}}=5,000,000$ datapoints by first sampling uniformly a vector $c_i$ from $\mathcal K$ for $i=1\dots M_{\text{train}}$, and then, for each $i$, we make an independent draw from the Brownian motion $W$ appearing in \eqref{X:x:t:s}. Based on this draw, we calculate
\begin{equation}
  \mathcal X(c_i):=\phi(X_{x,t}(T))\exp\left[\int_t^T c_i(X_{x,t}(s))\,ds\right].  
\end{equation}
for $x\in \{-1,-0.5,0,0.5,1 \} $. The set $(c_i,\mathcal X(c_i))$ for $i=1, \dots M_{\text{train}}$ constitutes the training set for each $x\in \{-1,-0.5,0,0.5,1 \} $. We train the neural network with $25$ epochs and a batch size of $10,000$. We denote the resulting neural network by $\mathcal{N}^x_\sigma$.

To test the accuracy of our network, we generate a test set of size $M_{\text{test}}=10,000$. For this, we first randomly sample functions $\tilde{c}_i,i=1,\dots, M_{\text{test}}$ from $\mathcal K$. For each of these samples we now calculate $u(x,t,\tilde{c}_i)$ for $x\in \{-1,-0.5,0,0.5,1 \}$ based on Monte Carlo simulation with $10,000$ paths. Note in passing that we include $\tilde{c}_i$ in the argument of $u$ to emphasize the dependency, slightly abusing the notation. We consider $(\tilde{c}_i,u(x,t,\tilde{c}_i)),i=1,\dots, M_{\text{test}}$ as examples from the ground truth. Next, we calculate the mean square error of the neural network predictions with respect to the ground truth given by 
\[
\frac{1}{M_{\text{test}}} \sum_{i=1}^{M_{\text{test}}} (\mathcal{N}_\sigma (\tilde{c}_i) - u(x,t,\tilde{c}_i))^2.
\]
 In Table \ref{tab:mse_comparison}, first column, we list the mean squared error for the different values of $x$. In Figure \ref{fig:boxplot_standard} we provide the resulting box plots and in Figure \ref{fig:hist_x_neg1}- \ref{fig:hist_x_1} the histograms of the distributions of the errors $\mathcal{N}_\sigma (\tilde{c}_i) - u(x,t,\tilde{c}_i) $. We stress that we essentially fix $5$ separate neural networks for each $x$ in this proof-of-concept case study. 

Next, for comparison, we fit a DeepONet structure to learn the map $c \mapsto u (t, x , c)$ via minimizing the energy functional \eqref{energy} based on the samples $(c_i,\mathcal X(c_i))$ for $i=1, \dots M_{\text{train}}$.  We use a 2-layer DeepONet with a branch net of $50$ nodes and a trunk net of $50$ nodes. We choose a ReLU (rectified linear unit) activation function. The DeepONet has $6,301$ parameters, comparable to the number of parameters of our Fr\'echet network ($6,500$ parameters). Because the DeepONet requires sampling of $c_i$ on a grid, we evaluate each $c_i$ from the training set on an equally spaced grid $\{y_1, \dots , y_{20} \}$ of size $20$. The training set for the DeepONet is composed of the set $\{ ((\bm{c}_i,x_j),u(t,x_j,c_i)) \}_{i=1, \dots M_{\text{train}}, 
 1\leq j\leq 5}$ 
where $\bm{c}_i=(c_i(y_1),\dots, c_i(y_{20}))$ and $(x_1,x_2,x_3,x_4,x_5)=(-1,-0.5,0,0.5,1)$. We train the network again with $25$ epochs and a batch size of $10,000$. We denote the resulting neural network by $\mathcal{N}^{DON}_\sigma$. Now we take the same test set as before and evaluate it on the grid, i.e., calculate $\{((\tilde{\bm{c}}_i,x_j),u(t,x_j,c_i))\}_{i=1, \dots M_{\text{test}},1\leq j\leq 5}$. The mean squared errors are presented in the right column of Table \ref{tab:mse_comparison}. The error distributions are displayed as a box plot in Figure\ref{fig:boxplot_DON} and as histograms in Figures \ref{fig:DON_hist_x_neg1}-\ref{fig:DON_hist_x_1}. 

Overall we observe an error of similar magnitude for both architectures. For all values of $x$ the mean square error is slightly lower for the Fr\'echet neural network, except for $x=0.5$, where it is lower for DeepONet. We further observe that the error distribution for the Fr\'echet neural network is more symmetric while for DeepONet it is heavily skewed. We stress that the DeepONet is basically trained on a training set $5$ times larger than the training set for the Fr\'echet neural network. This is due to the nature of DeepONet approximating the map $c \mapsto u (t, \cdot , c)$, i.e., it learns the entire solution function $u (t, \cdot , c)$. This requires to feed in the argument $x$, at which $u (t, \cdot , c)$ is to be evaluated, resulting in the training set $\{ ((\bm{c}_i,x_j),u(t,x_j,c_i)) \}_{i=1, \dots M_{\text{train}}, 
 1\leq j\leq 5}$.  In contrast to Fr\'echet neural network which is separately trained for each $x\in\{-1,-0.5,0,0.5,1\}$, DeepONet can therefore make use of information across different values of $x$ in the learning process. 

As evident from the theory covered in Section \ref{NN:Intro} and \ref{sect:cont}, it is possible to learn the entire solution $u (t, \cdot , c)$ with the Fr\'echet neural network structure. In fact, we know from Section \ref{sect:cont} that the non-linear and continuous operator $c \mapsto u (t, \cdot , c)$ is a continuous operator from $W^{m,p}$ to $C(K)$. Because $C(K)$ naturally embeds continuously into $L^2(K)$, we can actually see this operator as an operator from $W^{m,p}$ to $L^2(K)$ where it is still continuous by composition. We can now choose any orthonormal basis in $L^2(K)$ to represent the solution $u (t, \cdot , c)$. For example if $K = [0,1]^n$, then we can use tensor products of sinus and cosinus functions. We can compute these Fourier coefficients simply by evaluating numerical integrals, without computing derivatives. With the Fourier coefficients and the sinus and cosinus basis functions, structural information about the solution $u (t, \cdot , c)$ could be used. We expect a significant improvement in the learning for Fr\'echet neural networks in this case. We leave a full scale numerical analysis of this approach for future research.

\begin{table}[h]
    \centering
    \begin{tabular}{|c|c|c|}
        \hline
        $x$ Value & Fr\'echet NN & DeepONet \\
        \hline
        $x=-1$ & 0.011 & 0.047 \\
        $x=-0.5$ & 0.040 & 0.147 \\
        $x=0$ & 0.061 & 0.069 \\
        $x=0.5$ & 0.042 & 0.023 \\
        $x=1$ & 0.016 & 0.033 \\
        \hline
    \end{tabular}
    \caption{Mean Squared Error for various $x$ values across two methods, rounded to $10^{-3}$.}
    \label{tab:mse_comparison}
\end{table}

\begin{figure}[t]
    \centering
    \subfigure[Box plot Error Distribution (Fr\'echet NN)]{\includegraphics[width=0.45\textwidth]{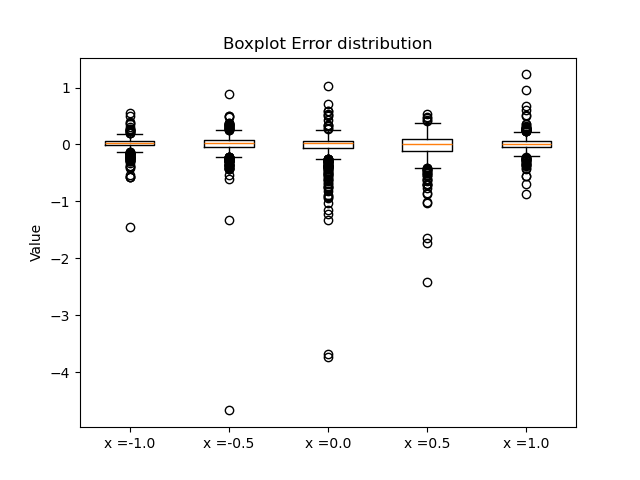} \label{fig:boxplot_standard}}
    \hfill
    \subfigure[Box plot Error Distribution (DON)]{\includegraphics[width=0.45\textwidth]{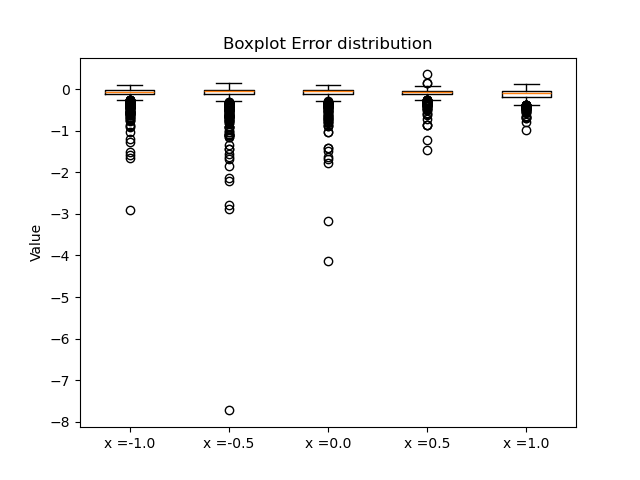} \label{fig:boxplot_DON}}
    
    \caption{Box plots of the error distributions. The left figure shows the result from the Fr\'echet neural network, and the right figure the results with DeepONet.}
    \label{fig:boxplot_comparison}
\end{figure}

\begin{figure}[t]
    \hfill
    \subfigure[$x=-1$]{\includegraphics[width=0.18\textwidth]{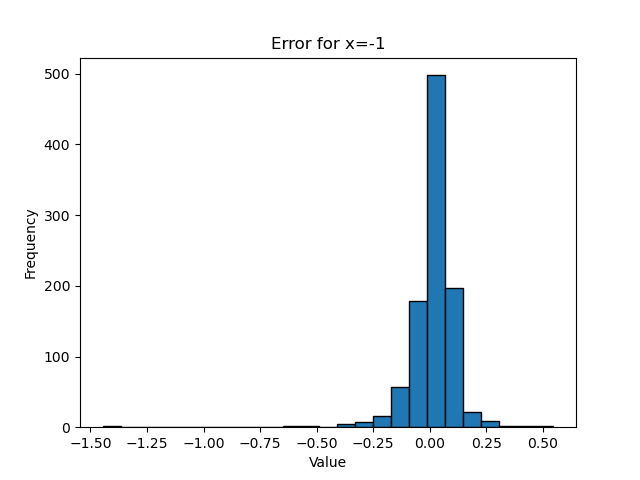} \label{fig:hist_x_neg1}}
    \hfill
    \subfigure[$x=-0.5$]{\includegraphics[width=0.18\textwidth]{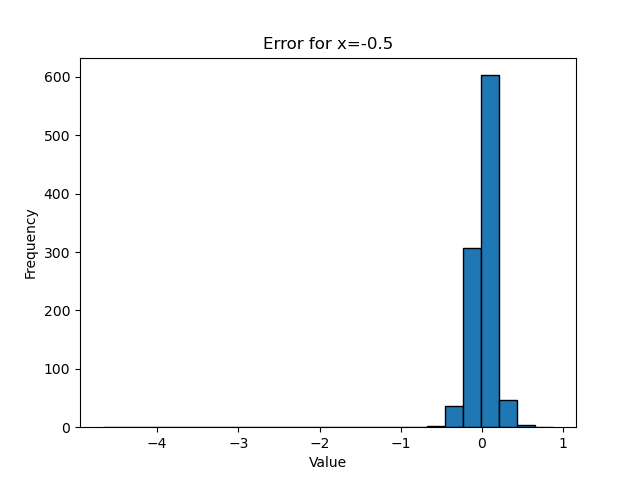} \label{fig:hist_x_neg0.5}}
    \hfill
    \subfigure[$x=0$]{\includegraphics[width=0.18\textwidth]{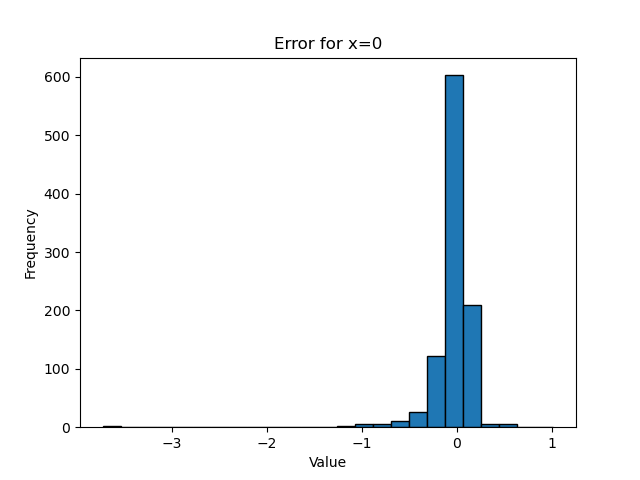} \label{fig:hist_x_0}}
    \hfill
    \subfigure[$x=0.5$]{\includegraphics[width=0.18\textwidth]{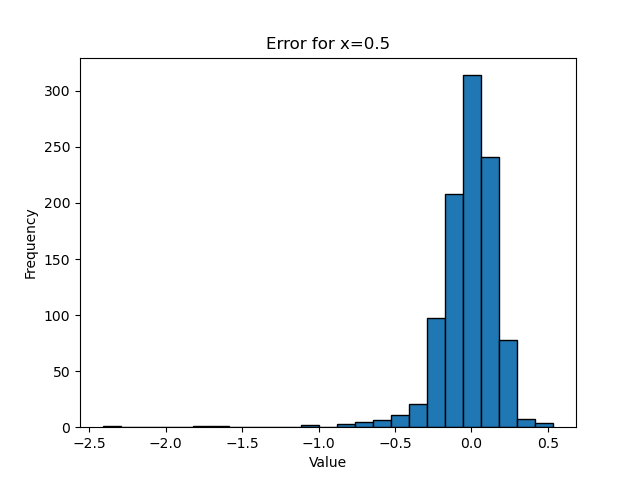} \label{fig:hist_x_0.5}}
    \hfill
    \subfigure[$x=1$]{\includegraphics[width=0.18\textwidth]{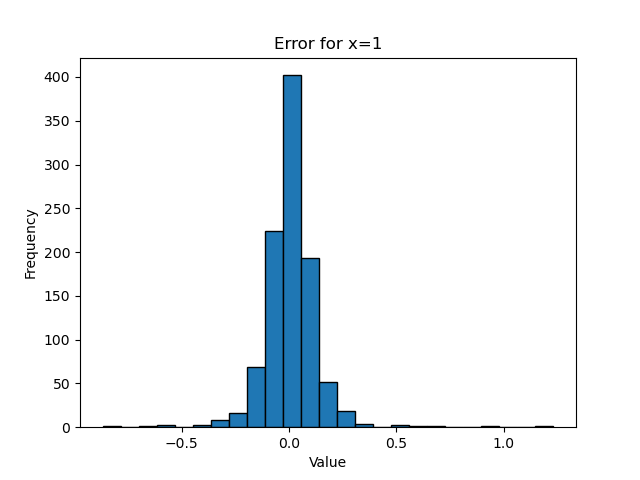} \label{fig:hist_x_1}}
    \hfill
    
    \vspace{1em} 
    \hfill
    \subfigure[$x=-1$]{\includegraphics[width=0.18\textwidth]{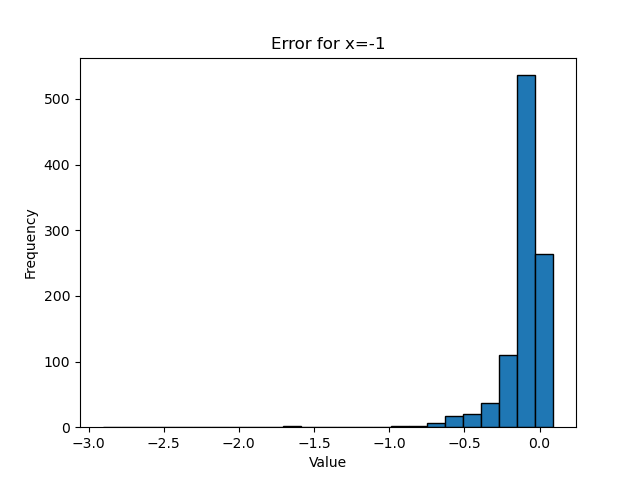} \label{fig:DON_hist_x_neg1}}
    \hfill
    \subfigure[$x=-0.5$]{\includegraphics[width=0.18\textwidth]{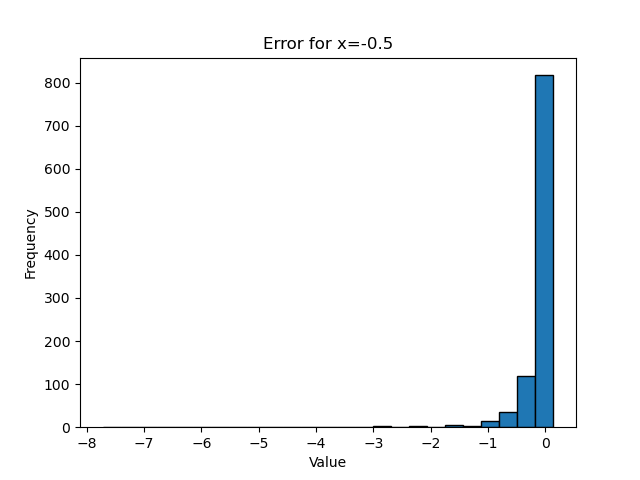} \label{fig:DON_hist_x_neg0.5}}
    \hfill
    \subfigure[$x=0$]{\includegraphics[width=0.18\textwidth]{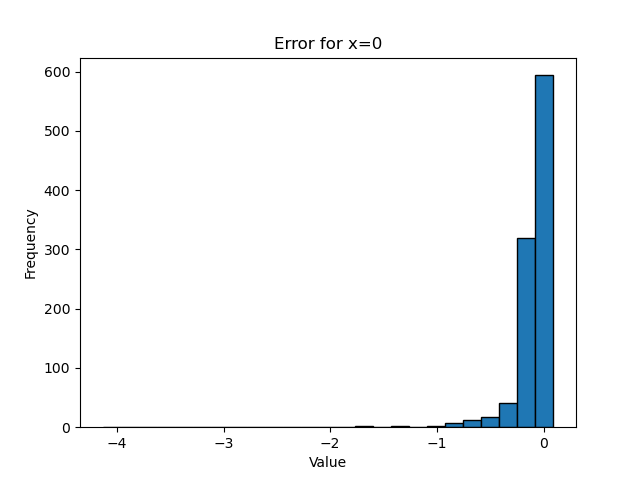} \label{fig:DON_hist_x_0}}
    \hfill
    \subfigure[$x=0.5$]{\includegraphics[width=0.18\textwidth]{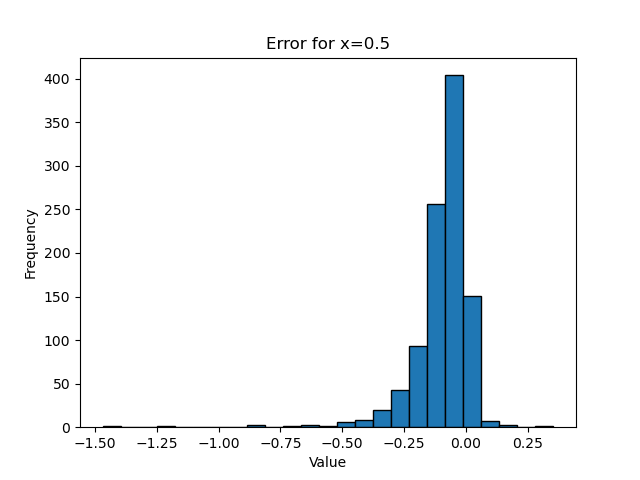} \label{fig:DON_hist_x_0.5}}
    \hfill
    \subfigure[$x=1$]{\includegraphics[width=0.18\textwidth]{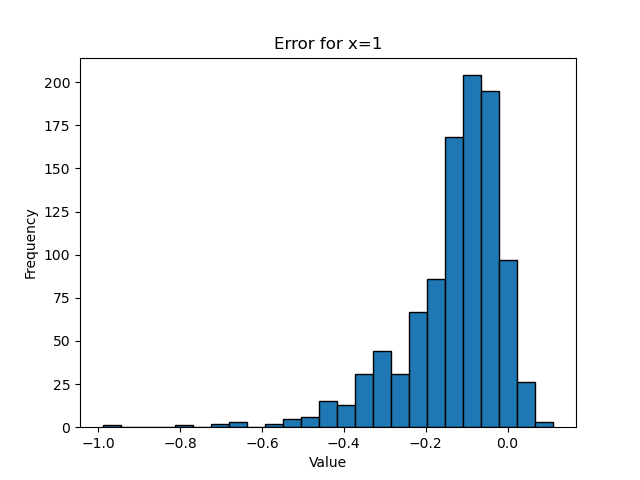} \label{fig:DON_hist_x_1}}
    \hfill

    \caption{Comparison of error distributions for various values of $x$. The first row shows the result from the Fr\'echet neural network, and the second row shows the results with DeepONet.}
    \label{fig:histograms_comparison}
\end{figure}

Overall the proof-of-concept case study presented here shows already that using the structural information, being the key in our proposed Fr\'echet neural network architecture, is promising. We believe that our approach can be applied successfully in  many situations where an entire set of partial differential equations needs to be solved at once. For example, in mathematical finance one is interested in pricing derivatives like options (see e.g. \cite{bjork}). 
With $X$ being a model for the market and $\phi$ signifying the payoff function of a derivative at time $T$, the price dynamics of the derivative can be described by $u(t,x;\phi)$ with $x=X(t)$. By learning the operator map $\phi\mapsto u(t,x;\phi)$, one has available a pricing generator for such derivatives and portfolios thereof. The function $c$ is mapping the multivariate process $X$ into an interest rate dynamics in such a context, and from derivatives prices one can consider the inverse problem of re-constructing $c$ from data. Having access to the operator map $(\phi,c)\mapsto u(t,x;\phi,c)$ and its structural representation provides a tool for solving this problem. Moreover, we can price portfolios of derivatives after learning the operator map. 
A damped $L^2$-space is an appropriate space for payoff functions, i.e., $L^2(w)$ with $w(dx)=\exp(-x^2)dx$. Another important problem in option theory is computing the implied volatility. This entails in recovering the covariance function $a$ (the matrix specifying the elliptic operator in \eqref{eq: elliptic operator}) from knowing the prices, given by $u$. I.e., this is the inverse problem for the operator map $a\mapsto u(t,x;a)$. If one is able to specify a suitable space for $a$ as well as showing continuity of the operator map, we can use our framework for this task.    

Random parabolic partial differential equations is another avenue of applications of our operator-learning methodology (see e.g. \cite{NabianMeidani}, who propose a deep neural network architecture to solve high-dimensional random partial differential equations). If one or more of the parameter functions $\phi, c$ or $f$ are random, by knowing the operator $(\phi,c,f)\mapsto u(t,x;\phi,c,f)$ we can efficiently sample from the solution $u$ by drawing random samples of the input functions. By representing both the input functions and the output map in terms of their basis functions, we are indeed sampling the loadings of the basis expansion of the input, and using the learned network for the output loadings. We believe this is a fruitful approach for uncertainty quantification, in particular for high dimensional problems.

\bibliographystyle{agsm}
\bibliography{bio}  

\end{document}